\documentclass[12pt]{article}

      \usepackage{latexsym}
         \usepackage[reqno, namelimits, sumlimits]{amsmath}
         \usepackage{amssymb, amsfonts}
         \usepackage{amsthm}
           
 \newtheorem{theorem}{Theorem}[section]
 
 \newtheorem{definition}[theorem]{Definition}
 
 \newtheorem{lemma}[theorem]{Lemma}

 \newtheorem{pro}[theorem]{Proposition}

\author{ Gregory Seregin
}

\title{A Note On Local Regularity of Axisymmetric Solutions to the Navier-Stokes Equations}

\author{G.~Seregin\footnote{University of Oxford, Mathematical Institute, OxPDE, Oxford, UK and St Petersburg Department of Steklov Mathematical Institute, RAS, Russia, email address: \texttt{seregin@maths.ox.ac.uk}}
}

\begin{document}

\maketitle


\centerline{ To the memory of Olga Alexandrovna Ladyzhenskaya
}

\begin{abstract} In the paper, a new {\it slightly supercritical}
condition, providing {\it local} regularity of axially symmetric solutions to the non-stationary 3D Navier-Stokes equations, is discussed. It  generalises almost all known results in the local regularity theory of weak axisymmetric solutions.
\end{abstract}


{\bf Data availability statement}
Data sharing not applicable to this article as no datasets were generated or analysed during the current study.

{\bf Acknowledgement} 
 The work is supported by the grant RFBR 20-01-00397.

\pagebreak

\setcounter{equation}{0}
\section{Introduction }

In the submitted note, we study   
potential singularities of axially symmetric flows of viscous incompressible fluids. 
It has been shown in  the paper \cite{Seregin2020}  that  an axisymmetric solution to the three-dimensional non-stationary Navier-Stokes equations with a bounded scale-invariant energy quantity is actually smooth,  i.e., axially symmetric solutions do no exhibit Type I blowups (see the paper \cite{Seregin2020} for more definitions and explanations).  

Let us consider the 3D non-stationary Navier-Stokes system
\begin{equation}\label{NSE}
\partial_tv+v\cdot\nabla v-\Delta v=-\nabla q,\qquad {\rm div}\,v=0	
\end{equation}
in the  space-time parabolic cylinder $Q=\mathcal C\times]-1,0[$, where $\mathcal C=\{x=(x_1,x_2,x_3):\,x_1^2+x^2_2<1,\,-1<x_3<1\}$ is a usual right circle cylinder in $\mathbb R^3$. In what follows, it is supposed that the pair $v$ and $q$ is the so-called {\it suitable weak 
  solution} to the Navier-Stokes equations in 
   $Q$. We recall a definition of such solutions.
\begin{definition}\label{sws}
Let $\omega\subset \mathbb R^3$ and $T_2>T_1$. The pair $w$ and $r$ is a suitable weak solution to the Navier-Stokes system in $Q_*=\omega\times ]T_1,T_2[$ if:

1. $w\in L_{2,\infty}(Q_*)$, $\nabla w\in L_2(Q_*)$, $r\in L_\frac 32(Q_*)$;

2. $w$ and $r$ satisfy the Navier-Stokes equations in $Q_*$ in the sense of distributions;

3. for a.a. $t\in [T_1,T_2]$, the local energy inequality
$$\int\limits_\omega\varphi(x,t)|w(x,t)|^2dx+2
\int\limits_{T_1}^t\int\limits_\omega\varphi|\nabla w|^2dxdt’\leq\int\limits_{T_1}^t\int\limits_\omega[|w|^2(\partial_t\varphi+\Delta \varphi)+$$
$$+w\cdot\nabla\varphi(|w|^2+2r)]dxdt’$$	holds for all non-negative $\varphi\in C^1_0(\omega\times ]T_1,T_2+(T_2-T_1)/2[).$
\end{definition}

Our standing assumption is that a suitable weak solution $v$ and $q$ to the Navier-Stokes equations in $Q=\mathcal C\times ]-1,0[$ is axially symmetric with respect to $x_3$-axis. It means the following: if we introduce the corresponding cylindrical coordinates $(r,\theta,x_3)$ and use the corresponding representation $v=v_r e_r+v_\theta e_\theta+v_3e_3$, then $v_{r,\theta}=v_{\theta,\theta}=v_{3,\theta}=q_{,\theta}=0$. Here, the comma in lower indices means the partial derivative in the indicated spatial direction.

With the  regards to the state of arts in the regularity theory of axially symmetric solutions to the Navier-Stokes equations, we could refer to the previous papers \cite{Seregin2020} and \cite{Seregin2021} of the author and especially to references therein. For example, one could mention the following very interesting papers:
 \cite{L1968}, \cite{UY1968},
\cite{LMNP1999}, \cite{NP2001}, \cite{Po01}, \cite{CL2002}, \cite{SZ2007}, \cite{ChenStrainYauTsai2009},
\cite{SS2009},  \cite{LeiZhang2011}, \cite{Pan16},\cite{LeiZhang2017}, \cite{ChenFangZhang2017},  \cite{SerZhou2018}, and \cite{ZhangZhang2014}.

In a sense, our note  is a continuation of  author's paper \cite{Seregin2021}, which, in turn, has been inspired   by  paper \cite{Pan16} of X. Pan, where the regularity of solutions has been proved under a slightly supercritical assumption. The aims of the present paper are to consider a local setting and  a different supercritical assumption.

In order to describe our supercritical assumption, additional notation is needed. Given $x=(x_1,x_2,x_3)\in \mathbb R^3$, denote $x'=(x_1,x_2,0)$. Next, different types of spatial cylinders will be denoted as $\mathcal C(r)=\{x:\,|x’|<r, |x_3|<r\}$, $\mathcal C(x_0,r)=\mathcal C(r)+x_0$, $Q^{\lambda,\mu}(r)=\mathcal C(\lambda r)\times]-\mu r^2,0[$, $Q^{1,1}(r)= Q(r)$, $Q^{\lambda,\mu}(z_0,r)=\mathcal C(x_0,\lambda r)\times]t_0-\mu r^2,t_0[$.  And, finally,  let
$$
f(R):=\frac{1}{\sqrt R} \Big(\int\limits^0_{-R^2}\Big(\int\limits_{\mathcal C(R)}	|\overline v|^3dx\Big)^\frac 43dt\Big)^\frac 34
$$
and 
$$
	M(R):=\frac{1}{\sqrt R}\Big(\int\limits_{Q(R)}|\overline  v|^\frac {10}3dz\Big)^\frac 3{10}$$
for any $0<R\leq  1$, with $\overline v=v_re_r+v_3e_3$, and  assume that:
\begin{equation}
	\label{sca1}
	f(R)+M(R)\leq g(R):=c_*\ln^\alpha\ln^\frac 12(1/R)
\end{equation}
for all $0<R\leq 2/3$, where $c_*$ and $\alpha$ are positive constants and $\alpha$ obeys the condition:
\begin{equation}
	\label{sca2} 0<\alpha\leq \frac 1{224}.
\end{equation}

In the paper \cite{Seregin2021},   the following completely local result has been stated. 
 \begin{theorem}
\label{mainresult}
Assume that a pair $v$ and $q$	is an axially symmetric suitable weak solution to the Navier-Stokes equations in $Q$ 
and let conditions (\ref{sca1}) and (\ref{sca2})
hold. Then
the origin $z=0$ is a regular point of $v$. 
\end{theorem} 

Unfortunately, the above theorem  has not been proven in \cite{Seregin2021}. Instead,  a  global version of it has been stated and demonstrated there. However, a big step toward a proof of Theorem \ref{mainresult} has been made in that paper \cite{Seregin2021}.
It is a careful analysis of the scalar equation
\begin{equation}
	\label{eqsigma}
	\partial_t\sigma+\Big(v+2\frac{x'}{|x'|^2}\Big)\cdot\nabla\sigma-\Delta\sigma=0,
\end{equation}
being held in domain $Q\setminus(\{x'=0\}\times]-1,0[)$. Here,
 $\sigma :=\varrho v_\varphi=v_2x_1-v_1x_2$.
 
 Let us recall known differentiability properties of $\sigma$. Most of them follow from the partial regularity theory developed by Caffarelli-Kohn-Nirenberg in their famous   paper \cite{CKN}. As to further discussions and improvements,  see also papers \cite{Lin} and  \cite{LS1999} and of course references therein.

 Indeed,
 since $v$ and $q$ are an axially symmetric suitable weak solution in $Q$, there exists a closed subset $S^\sigma$ of $Q$, whose 1D-parabolic measure in $\mathbb R^3\times \mathbb R$ is equal to zero and $x'=0$ for any $z=(x,t)\in S^\sigma$, such that any spatial derivative of $v$ with respect to Cartesian coordinates  (and thus of $\sigma$) is H\"older continuous in 
 $Q\setminus S^\sigma$ in space-time.
 
 Next, 
$$\sigma\in W^{2,1}_p(P(\delta,R;R)\times ]-R^2,0[)$$
for any $0<\delta <R<1$ and for any finite exponent $p\geq 2$. 


One can show also that
$\sigma\in L_\infty(Q(R))$ for any $0<R<1$, see, for example, papers \cite{SS2009} and \cite{Seregin2020}.

What actually has been proved in  paper \cite{Seregin2021}, see also   the last section of the present paper, is as follows:
\begin{pro}
\label{modcon} Assume that all assumptions of Theorem \ref{mainresult} hold.
 Let $\sigma=\varrho v_\varphi$, then
  \begin{equation}
\label{osc}
{\rm osc}_{z\in Q(r)}\leq e^{-c\Big[\ln^\frac 14(1/(2r))
-\ln^\frac 14(1/(2R))\Big]}
{\rm osc}_{z\in Q(2R)}\sigma(z)	\end{equation} for all $0<r<R\leq  R_*(c_*,\alpha)\leq 1/6$, 	where $c$ is a positive absolute constant. 

Here, ${\rm osc}_{z\in Q(r)}\sigma(z) =M_{r}-m_{r}$ and
$$M_{r}=\sup\limits_{z\in Q(r)}\sigma(z),\qquad m_{r}=\inf\limits_{z\in Q(r)}\sigma(z).$$ 	\end{pro}
Now, the main task of the note is to deduce Theorem \ref{mainresult} from Proposition \ref{modcon}.

Let us describe briefly  the most important counter-parts of our arguments. 
 The first one is a choice of a cut-off function that takes into account  the well known partial regularity of axisymmetric flows. The second important point is a local regularity properties of solution $\hat v$ to the following elliptic system:
\begin{equation}
	\label{ellipticsystem}
	 {\rm div}\,\overline v=0,\qquad
{\rm curl} \,\overline v=\omega_\theta e_\theta
\end{equation}
in $\mathcal C$, where $\omega_\theta=v_{r,3}-v_{3,r}$ is the corresponding component of the vorticity field $\omega={\rm curl}\, v$, see Lemma \ref{equforvbar}  of Section 2. It should be mentioned that Lemma \ref{equforvbar} of the present note is a local version of the corresponding global statement, see \cite{MiaoZheng2013}, Proposition 2.5 therein, or  \cite{ChenFangZhang2017}, Lemma 2.3 therein.

 In Section 3, we show how  Pan type results on supercritical regularity, see Pan's paper \cite{Pan16}, can be    deduced  
  from Theorem \ref{mainresult} of the present note. To be a bit more precise, let us replace the main supercritical condition (\ref{sca1})  with  the other one, where, for simplicity, it is assumed that:
\begin{equation}\label{Pan}
	|v(r,x_3,t)|\leq \frac cr\ln^\alpha \ln\frac 1r
\end{equation}
in $Q$, with a positive number $\alpha$. 
\begin{theorem}
	\label{PanType} Assume that $v$ and $q$ are suitable weak solution to the Navier-Stokes equations in $Q$ and condition (\ref{Pan}) is satisfied  with  $0<\alpha<4/7$. Then, condition (\ref{sca1}) holds as well but with new exponent $35\alpha/6$ and a new constant $c$, depending on $\alpha$, and some energy quantities of $v$ and $q$.
\end{theorem}

As it follows from Theorem \ref{mainresult}, supercritical condition (\ref{Pan}) implies regularity of the velocity field $v$ at the origin $z=0$ for sufficiently small $\alpha$.

 At the end of the section, we would like to comment our notation. All absolute constants are denoted by $c$, other constants are denoted by $C$ with indication of variables of  which those constants may   depend on. The norms of the Lebesgue space $L_p(\omega)$ are denoted by $\|\cdot\|_{p,\omega}$ and of the mixed Lebesgue space $L_{p,q}(Q_T)=L_q(0,T;L_p(\omega))$ are denoted by $\|\cdot\|_{p,q,Q_T}$, where $Q_T=\omega\times ]0,T[$.


\setcounter{equation}{0}
\section{Proof of Theorem \ref{mainresult}}

{\bf Step 1}. {\it  Construction of a cut-off function}. The partial regularity theory
for the Navier-Stokes equations implies that  if singular points of an axisymmetric velocity field $v$ exist,  they must belong to the axis of symmetry, which is $x_3$-axis in our case. Since the 1D parabolic Hausdorff measure of the set of singular points is equal to zero,  there exist at least two regular points $z_1=(x',x_3,t)=(0,h_1,0)$ and $z_2=(0,-h_2,0)$ of $v$ such that $0<h_1,h_2<1$. 
According to the properties of regular points, there are cylinders $Q_1=Q(z_1,\delta)$ and 
$Q_2=Q(z_2,\delta)$ with radius $\delta>0$ such that
$v\in C([-\delta^2,0];C^3(\overline{\mathcal C}((0,h_1),\delta)))\cap C([-\delta^2,0];C^3(\overline{\mathcal C}((0,-h_2),\delta)))$ at least.  Moreover, one can find $t_0\in ]-\delta^2,0[$ such that there is no singular point in the set $\overline {\mathcal C}(r_0)\times [t_0,t_0+\delta^2_0]\subset ]-t_0,0[$ with $0<r_0<1$ and $\delta_0>0$. It is also possible to pick up $r_0$ close to one so that $r_0\geq \max\{h_1-\delta,h_2-\delta\}$.

 Without loss of generality (just for simplicity of calculations),  we may assume that $h_1=h_2=3/4$, $\delta=1/8$, and $t_0=-1/32$, $t_0+\delta_0=-1/64$, $r_0=4/5$. Then we are going  to use a smooth cut-off function $\eta(r,x_3,t)=\varphi(r)\psi (x_3)\xi(t)$, where $\varphi$, $\psi$, and $\xi$  
are smooth functions having the following properties: 
\begin{itemize}
\item all of them take the values in $[0,1]$;

\item $\varphi(r)=1$ if $0\leq r<1/2$;  $\varphi(r)=0$ if $r\geq 5/6$;
 
\item $\psi(x_3)=1$ if $0\leq x_3<1/2$; $\psi(x_3)=0$ if $x_3\geq 5/6$;
 
 \item $\xi(t)=1$ if $t\geq -1/64$; $\xi(t)=0$ if  $t\leq -1/32$.
\end{itemize}

Without loss of generality, we  may assume also that
\begin{equation}\label{sigmaest}
|\sigma(r,x_3,t)|=r|v_\theta(r,x_3,t)|\leq C_1e^{-c\ln^\frac 14(1/r)}
\end{equation}
in $\mathcal C$. It is easy to see that (\ref{sigmaest}) implies
\begin{equation}\label{sigmaest2}
|\sigma(r,x_3,t)|=r|v_\theta(r,x_3,t)|\leq \frac {cC_1}{ln^3(e/r)}
\end{equation}
for the same $x$ and $t$.

\vskip 0.5cm

\noindent
{\bf Step 2}. {\it Auxiliary Lemma.} \begin{lemma}
	\label{important} Let $\overline v=v_re_r+v_3e_3$, where $v$ is our suitable weak solution, and let $\Gamma=\omega_\theta/r$. Then, for all $t\in ]-1,0[$,  the following estimates are valid:
	$$\|\nabla (\eta^3v_r/r)(\cdot,t)\|_{2,\mathcal C}\leq c\|\eta^3\Gamma(\cdot,t)\|_{2,\mathcal C}+C(v,\eta) 
	$$ and 
	$$\|\overline\nabla^2 (\eta^3v_r/r)(\cdot,t)\|_{2,\mathcal C}\leq c\|\eta^3\Gamma_{,3}(\cdot,t)\|_{2,\mathcal C}+C(v,\eta).
	$$  Here, $|\overline \nabla^2 g|^2=|g_{,rr}|^2+2|g_{,r3}|^2+|g_{,33}|^2$.
	\end{lemma}
\begin{proof} In order to prove the lemma, let us make use of elliptic system (\ref{ellipticsystem}), 
where time $t\in ]-1,0[$ is considered as a parameter. Introducing  $\zeta=\eta^3$, 
 derive	from (\ref{ellipticsystem}) the non-homogeneous  system:
$$ {\rm div}\,(\zeta\overline v)=\overline v\cdot\nabla \zeta,\qquad
{\rm curl} \,(\zeta\overline v)=\zeta \omega_\theta e_\theta+\nabla \zeta\times\overline v$$
in $\mathbb R^3$. In turn, from the above system,  it follows that: 
\begin{equation}\label{equforvbar}
-\Delta (\zeta \overline v)={\rm curl}(\zeta \omega_\theta e_\theta) -\nabla (\overline v\cdot \nabla \zeta)
+{\rm curl}(\nabla \zeta\times\overline v).\end{equation}

Next, as a result of simple transformations, the following equation for $v_r$ can be deduced from (\ref{equforvbar}):
$$-\Delta (\zeta v_r e_r)=-((\zeta\omega_\theta)_{,3}+g)e_r-( (\overline v\cdot\nabla\zeta)_{,1}, (\overline v\cdot\nabla\zeta)_{,2},0),$$
where $g= (\zeta_{,3}v_r-\zeta_{,r}v_3)_{,3}$.
The latter equation can be solved in the whole space $\mathbb R^3$ so that
$$\frac {v_r}r=-\frac 1re_r\cdot\Delta^{-1}((\zeta\omega_\theta)_{,3}e_r)+\frac 1re_r\cdot\Delta^{-1}(G_{,1},G_{,2},0)-$$$$-\frac 1re_r\cdot\Delta^{-1}( (\overline v\cdot\nabla\zeta)_{,1}, (\overline v\cdot\nabla\zeta)_{,2},0),$$
where 
$$G(r,x_3,t)=\int\limits_r^\infty g(\varrho,x_3,t)d\varrho$$
and the operator $\Delta^{-1}$ is defined by Newton potential, i.e.,
$$\Delta^{-1}h(x)=\frac 1{4\pi_0}\int\limits_{\mathbb R^3}\frac 1{|x-y|}h(y)dy$$
for smooth compactly supported functions $h$. Here, $\pi_0=3,14..$.

 So,
$$\frac {v_r}r=W+Z, 
$$
where
$$W=-\frac 1re_r\cdot\Delta^{-1}((\zeta\omega_\theta)_{,3}e_r),\qquad
Z=\frac 1rZ_{0,r},\qquad Z_0=
\Delta^{-1}(G-\overline v\cdot\nabla \zeta).$$
Function $f=G-\overline v\cdot \nabla\zeta$ is axisymmetric, so does function 
  $Z_0=\Delta^{-1}f$. Hence,
 according to Proposition 2.9 of paper \cite{HR2011},  
the following identity is valid:
$$Z=\sin^2\theta Z_{0,11}-2\cos\theta\sin\theta Z_{0,12}+\cos^2\theta Z_{0,22}.
$$ Next, exploiting the properties of singular integrals, one can derive two estimates:
$$\|\nabla Z\|_{2,\mathbb R^3}\leq c\|\nabla_{x'} f\|_{2,\mathbb R^3}$$
and 
$$\|\overline\nabla^2 Z\|_{2,\mathbb R^3}\leq c\|\nabla^2_{x'} f\|_{2,\mathbb R^3}. $$
 Again, it is important to notice that $f(r,x_3,t)\neq0$ only if $|\nabla \eta(r,x_3,t)|>0$. In the set  ${\rm supp}\, |\nabla \eta|$ , functions  $v$, $\nabla v$, and $\nabla^2 v $  are bounded in space-time. The most dangerous term on the right hand  side of the latter inequalities are those where cut-off function $\eta$ does not contain derivatives in $r$. These  terms contain $v_3$, $v_{r,3}$, $v_{3,r}$, $v_{3,rr}$, $v_{3,r}/r$, $v_{r,3 r}$,  and $v_{r,3}/r$. All of them are bounded by either $|v|$, or $|\nabla v|$, or $|\nabla^2v|$.
So, $|\nabla Z_0(\cdot,t)|+|\overline\nabla^2 Z_0(\cdot,t)|\leq C(v,\eta)$.

As to the function $W$, the global estimates  have been already established in \cite{ChenFangZhang2017}. Here, they are:
$$\|\nabla W\|_{2,\mathbb R^3}\leq c \|\zeta \Gamma\|_{2,\mathbb R^3},\qquad \|\overline\nabla^2 W\|_{2,\mathbb R^3}\leq c \|(\zeta \Gamma)_{,3}\|_{2,\mathbb R^3}.$$
This completes the proof of the lemma.
\end{proof}

\vskip 0.5cm
\noindent
{\bf  Step 3.} {\it  Local estimates of solutions.} In this section, our goal is to make arguments of  papers \cite{Wei2016}    and  \cite{ChenFangZhang2017} completely local. 

 It is easy to verify that functions $\Phi=\omega_r/r=-v_{\theta,3}/r$ and $\Gamma=\omega_\theta/r=(v_{r,3}-v_{3,r})/r$ satisfy the following equations:
$$\partial_t\Phi+\Big(v-\frac {2x'}{|x'|^2}\Big)\cdot\nabla \Phi-\Delta\Phi+\omega\cdot\nabla\Big(\frac {v_r}r\Big)=0,
$$
$$
\partial_t\Gamma+\Big(v-\frac {2x'}{|x'|^2}\Big)\cdot\nabla \Gamma-\Delta\Gamma+2\frac {v_\theta}r\Phi=0.
$$
Let us multiply the first equation by $\Phi\eta^6$ and the second equation by $\Gamma \eta^6$. After integration by parts, we find: 
$$\frac 12\partial_t\int\limits_\mathcal C(\Phi \eta^3)^2dx+\int\limits_\mathcal C(\eta^3|\nabla \Phi|)^2dx+\pi_0\int^1_{-1}(\eta^3\Phi)^2|_{x'=0}dx_3=$$
$$
=\frac 12\int\limits_\mathcal C\Phi^2(\partial_t\eta^6+\Delta\eta^6)+\int\limits_\mathcal C\Big(v-\frac {2x'}{|x'|^2}\Big)\cdot\nabla\eta^6 \Phi^2dx+$$
$$+\int\limits_\mathcal C\Big(v_\theta\Big(\frac {v_r}r\Big)_{,3}(\Phi\eta^6)_{,r}-v_\theta\Big(\frac {v_r}r\Big)_{,r}(\Phi\eta^6)_{,3}\Big)dx=A_1+A_2+A_3$$
and 
$$\frac 12\partial_t\int\limits_\mathcal C(\Gamma \eta^3)^2dx+\int\limits_\mathcal C(\eta^3|\nabla \Gamma|)^2dx+\pi_0\int^1_{-1}(\eta^3\Gamma)^2|_{x'=0}dx_3=$$
$$
=\frac 12\int\limits_\mathcal C\Gamma^2(\partial_t\eta^6+\Delta\eta^6)+\int\limits_\mathcal C\Big(v-\frac {2x'}{|x'|^2}\Big)\cdot\nabla\eta^6 \Gamma^2dx-$$
$$-2\int\limits_\mathcal C\frac {v_\theta}r\Phi\eta^6\Gamma dx=B_1+B_2+B_3.$$ 

Now, we wish to evaluate quantities $A_i$ and $B_i$, starting  with $A_1$ and $B_1$ and letting $ \Psi=\partial_t\eta^6+\Delta\eta^6$ to  simplify  our notation. Indeed, by the construction of our cut-off function $\eta$, the solution $v$ is smooth in the domain where $|\Psi|>0$.
Now, it remains to use inequality $|\Gamma|^2 +|\Phi|^2\leq |\nabla \omega|^2$ and boundedness of $|\nabla \omega|$ in the corresponding space-time domain. So, we have
$$A_1+B_1= \int\limits_{\mathcal C}(|\Phi|^2+|\Gamma|^2)\Psi dx\leq  \int\limits_{\mathcal C}|\nabla \omega|^2|\Psi| dx\leq C(v,\eta).$$

In order to estimate $A_2$ and $B_2$, boundedness of $v$ and its spatial derivatives   over the support of $\nabla \eta$  and the obvious identity
$$\frac {x'}{|x'|^2}\cdot\nabla \eta=\frac 1{|x'|}(\eta)_{,r}\leq C(\eta)$$
are used. So, we have 
$$A_2+B_2\leq C(v,\eta)\int\limits_{\frac 12<|x'|<1}\eta^5(|\Gamma|^2+|\Phi|^2)dx.$$
%
Since $|\Gamma|^2+|\Phi|^2\leq |\nabla\omega|^2$ and $| \nabla\omega|$ is bounded in $\{\frac 12<|x'|<1\}$,  the  estimate
$$A_2+B_2\leq C(v,\eta)
$$ is easily derived.

Our next goal is to find bounds for  $A_3$ and $B_3$. To this end, let us fix a number $0<r_1<1/4$ and introduce domain $S_1=\{x\in\mathcal C:\, |x'|<r_1\}$.
Then,
$$B_3= -2\int\limits_{S_1}\frac{v_\theta}r(\eta^3\Phi)(\eta^3\Gamma)+\frac 1{r_1}C(v,\eta)\leq $$$$\leq \frac {cC_1}{\ln (e/r_1)}\Big(\int\limits_\mathcal C \frac 1{r^2\ln^{2}(e/r)}|\eta^3\Gamma|^2dx\Big)^\frac 12\Big(\int\limits_\mathcal C \frac 1{r^2\ln^{2}(e/r)}|\eta^3\Phi|^2dx\Big)^\frac 12+$$
$$+\frac 1{r_1}C(v,\eta).$$
In order to estimate the right hand side of the latter inequality, we are going to use a Leray type inequality in dimension two.
For the reader convenience, let us state it for our particular case. The proof can be done with the help of integration by parts.
\begin{lemma}
	\label{Leray} For any function $f\in C^1_0(\mathcal C)$, the following inequality is valid:
$$\int\limits_{\mathcal C}\frac {|f|^2}{|x'|^2\ln^{2}(e/|x'|)}dx\leq  4\int\limits_{\mathcal C}|\nabla_{x'} f|^2dx.$$	
	\end{lemma}  

	Applying Lemma \ref{Leray}, 
 we find 
$$B_3\leq \frac {cC_1}{\ln (e/r_1)}\|\nabla_{x'}(\eta^3\Gamma)\|_{2,\mathcal C}\|\nabla_{x'} (\eta^3\Phi)\|_{2,\mathcal C}+\frac 1{r_1}C(v,\eta)\leq$$
$$\leq \frac {cC_1}{\ln (e/r_1)}(\|\eta^3\nabla_{x'}\Gamma\|_{2,\mathcal C}+\|\Gamma\nabla_{x'}\eta^3\|_{2,\mathcal C})(\|\eta^3\nabla_{x'}\Phi\|_{2,\mathcal C}+\|\Phi\nabla_{x'}\eta^3\|_{2,\mathcal C})+$$$$+\frac 1{r_1}C(v,\eta).$$
Now,  let us exploit one more time the fact that our solution $v$ has bounded spatial derivatives of any order in the support of $\nabla\eta$ and inequality $|\Gamma|^2+|\Phi|^2=(|\omega_r|^2+|\omega_\theta|^2)/{r^2}\leq |\nabla \omega|^2$. So, 
$$B_3  \leq \frac {cC_1}{\ln (e/r_1)}\|\eta^3\nabla_{x'}\Gamma\|_{2,\mathcal C}\|\eta^3\nabla_{x'}\Phi\|_{2,\mathcal C}+$$$$+C(v,\eta,r_1)(\|\eta^3\nabla_{x'}\Gamma\|_{2,\mathcal C}+\|\eta^3\nabla_{x'}\Phi\|_{2,\mathcal C})+C(v,\eta,r_1).
$$

As to the last quantity $A_3$, we find
$$A_3=\int\limits_\mathcal C\Big(v_\theta\Big(\frac {v_r}r\Big)_{,3}\eta^3(\eta^3\Phi)_{,r}-v_\theta\Big(\frac {v_r}r\Big)_{,r}\eta^3(\eta^3\Phi)_{,3}\Big)dx+$$
$$+\int\limits_\mathcal C\Big(v_\theta\Big(\frac {v_r}r\Big)_{,3}\eta^3\Phi(\eta^3)_{,r}-v_\theta\Big(\frac {v_r}r\Big)_{,r}\eta^3\Phi(\eta^3)_{,3}\Big)dx=A_{31}+A_{32}.$$
The first term on the right hand side can be transformed  so that: 
$$A_{31}=\int\limits_\mathcal C\Big(v_\theta\Big(\frac {\eta^3v_r}r\Big)_{,3}(\eta^3\Phi)_{,r}-v_\theta\Big(\frac {\eta^3v_r}r\Big)_{,r}(\eta^3\Phi)_{,3}\Big)dx-
$$
$$
-\int\limits_\mathcal C\Big(v_\theta\Big(\frac {v_r}r\Big)(\eta^3)_{,3}(\eta^3\Phi)_{,r}-v_\theta\Big(\frac {v_r}r\Big)(\eta^3)_{,r}(\eta^3\Phi)_{,3}\Big)dx=A_0+A'_{31}.$$
Since $|v_r/r|\leq |\nabla v|$ and since the second integral is taken over support of $\nabla \eta$, we find
$$A'_{31}\leq C(v,\eta)\|\nabla (\eta^3\Phi)\|_{2,\mathcal C}\leq C(v,\eta)(\|\eta^3\nabla\Phi)\|_{2,\mathcal C}+1).
$$
To evaluate $A_0$,  we may use H\"older inequality and the same trick as above. It gives  us to the bound: 
$$A_0\leq cC_1^2\Big(\int\limits_\mathbb C\frac {1}{r^{2}\ln^6(e/|x'|)}\Big(\Big|\Big(\frac{\eta^3v_r}r\Big)_{,3}\Big|^2+\Big|\Big(\frac{\eta^3v_r}r\Big)_{,r}
\Big|^2\Big)dx\Big)^\frac 12\|\nabla (\eta^3\phi)\|_{2,\mathcal C}\leq $$
$$\leq \frac {cC_1^2}{\ln^4(e/|r_1|)}\Big(\int\limits_\mathbb C\frac {1}{r^{2}\ln^2(e/|x'|)}\Big(\Big|\Big(\frac{\eta^3v_r}r\Big)_{,3}\Big|^2+$$$$+\Big|\Big(\frac{\eta^3v_r}r\Big)_{,r}
\Big|^2\Big)dx\Big)^\frac 12\|\nabla (\eta^3\phi)\|_{2,\mathcal C}+C(v,\eta,r_1).$$
Here, it has been used boundedness of $\nabla^2v$ and $\nabla^2 \omega$ in domain $\{r>r_1\}\times ]-1,0[$

Now, we again apply Lemma \ref{Leray} and find    
  another   bound: 
$$A_0\leq \frac {cC_1^2}{\ln^4(e/|r_1|)}(\|\nabla_{x'}
\Big(\frac{\eta^3v_r}r\Big)_{,3}\|^2_{2,\mathcal C}+\|\nabla_{x'}
\Big(\frac{\eta^3v_r}r\Big)_{,r}\|^2_{2,\mathcal C})^\frac 12\|\nabla (\eta^3\phi)\|_{2,\mathcal C}
+$$$$+C(v,\eta,r_1).$$
It remains to take into account the statement of Lemma \ref{important} and conclude 
$$A_0\leq  \frac {cC_1^2}{\ln^4(e/|r_1|)}(\|\eta^3\Gamma_{,3}\|_{2,\mathcal C}+C(v,\eta))(\|\eta^3\nabla\Phi\|_{2,\mathcal C}+C(v,\eta))+C(v,\eta,r_1)\leq $$
$$\leq  \frac {cC_1^2}{\ln^4(e/|r_1|)}\|\eta^3\nabla\Gamma\|_{2,\mathcal C}\|\eta^3\nabla\Phi\|_{2,\mathcal C}+$$$$+C(v,\eta,r_1)(\|\eta^3\nabla\Gamma\|_{2,\mathcal C}+\|\eta^3\nabla\Phi\|_{2,\mathcal C})+C(v,\eta,r_1).$$

Our next aim is a bound for $A_{32}$. Obviously, we have 
$$A_{32}=\int\limits_\mathcal C\Big(v_\theta\Big(\frac {\eta^3v_r}r\Big)_{,3}\Phi(\eta^3)_{,r}-v_\theta\Big(\frac {\eta v_r}r\Big)_{,r}\Phi(\eta^3)_{,3}\Big)dx.$$
Here, we would like to use again that $v$, $\nabla v$, and $\nabla^2v$ are bounded over the support of $\nabla \eta$. In addition,  we know that $|v_\theta|\leq |v|$,  $|\Phi|\leq \nabla \omega$, $|(v_r/r)_{,3}|\leq |\nabla^2v|$, and $|(v_r/r)_{,3}|\leq |\nabla^2v|$. Therefore, we find
$$A_{32}\leq C(v,\eta).$$

So, finally, 
$$
A_3
\leq  \frac {cC_1^2}{\ln^4(e/|r_1|)}\|\eta^3\nabla\Gamma\|_{2,\mathcal C}\|\eta^3\nabla\Phi\|_{2,\mathcal C}+$$$$+C(v,\eta,r_1)(\|\eta^3\nabla\Gamma\|_{2,\mathcal C}+\|\eta^3\nabla\Phi\|_{2,\mathcal C})+C(v,\eta,r_1).
$$

Combing all the estimates made on this step, we shall have:
$$\frac 12\partial_t\int\limits_\mathcal C(\Phi \eta^3)^2+(\Gamma \eta^3)^2dx+\int\limits_\mathcal C(\eta^3|\nabla \Phi|)^2+\eta^3|\nabla \Gamma|)^2dx\leq $$

$$\leq\Big(\frac {cC_1}{\ln (e/r_1)} +\frac {cC_1^2}{\ln^4(e/|r_1|)}\Big)\|\eta^3\nabla\Gamma\|_{2,\mathcal C}\|\eta^3\nabla\Phi\|_{2,\mathcal C}+$$$$+C(v,\eta,r_1)(\|\eta^3\nabla\Gamma\|_{2,\mathcal C}+\|\eta^3\nabla\Phi\|_{2,\mathcal C})+C(v,\eta,r_1).$$
Here, it is assumed that a number $r_1\in ]0,1/4[$ so small as
$$\frac {cC_1}{\ln (e/r_1)} +\frac {cC_1^2}{\ln^4(e/|r_1|)}<2.$$
So, if the latter condition holds,  the key estimate can be derived by more or less standard arguments. It is as follows:
$$\sup\limits_{-1<t<0}\int\limits_\mathcal C\eta^6(|\Gamma|^2+|\Phi|^2)dx+\int\limits_Q(\eta^3|\nabla \Phi|)^2+(\eta^3|\nabla \Gamma|)^2dx dt\leq C(v,\eta,r_1).$$

\vskip 0.5cm
\noindent
{\bf  Step 4.} {\it  Final Conclusion.} Now, let us show that the origin $z=0$ is a regular point of $v$. To this end, we are going to use the estimate proved at the previous step.  It can be re-written to the form: 
\begin{equation}\label{mainbound}
|\eta^3\Phi|^2_{2,Q}+|\eta^3\Gamma|^2_{2,Q}<\infty,	
\end{equation}
where as usual $|f|^2_{2,Q}=\sup\limits_{-1<t<0}\|f(\cdot,t)\|^2_{2,\mathcal C}+\|\nabla f\|^2_{2,Q}$. Since $\omega_\theta=r\Gamma$,
one can state that $|\eta^3\omega_\theta|_{2,Q}<\infty$ as well. It will be used to make various estimates of the solution to equation (\ref{equforvbar}). 
 Indeed, the elliptic theory implies two classical bounds:
$$\|\nabla (\eta^3\overline v)\|_{2,\mathcal C}
\leq c\|\omega_\theta\eta^3\|_{2,\mathcal C}+c\||\nabla \eta^3||\overline v|\|_{2,\mathcal C}$$
and
$$\|\nabla^2 (\eta^3\overline v)\|_{2,\mathcal C}\leq
c\||\nabla \eta^3||\nabla \overline v|\|_{2,\mathcal C}+c\||\nabla^2 \eta^3||\overline v|\|_{2,\mathcal C}+c\|{\rm curl} (\omega_\theta\eta^3e_\theta)\|_{2,\mathcal C}.$$
Let us notice that 
$${\rm curl} (\omega_\theta\eta^3e_\theta)=-(\omega_\theta\eta^3)_{,3}e_r+((\omega_\theta\eta^3)_{,r}+\Gamma\eta^3)e_3$$
and 
$$|\nabla \overline v|\leq |\nabla v|.$$ 
 Then,  our previous arguments can be exploited to describe properties of the solution $v$ in the support of $\nabla\eta$ and conclude that the quantity $|\nabla (\eta^3\overline v)|_{2,Q}$ is bounded that in turn yields boundedness of two norms: $\| \nabla (\eta^3\overline v)\|_{\frac {10}3,Q}$ and $\|\nabla (\eta^3\overline v)\|_{\infty,\frac {10}3,Q}$. So, for all sufficiently small $R$, we find
\begin{equation}\label{firstpart}
	\frac 1{R^2}\int\limits_{Q(R)}|\overline v|^3dz\leq cR^\frac{11}5\|\eta^3 \overline v\|^3_{\infty,\frac {10}3,Q}\to 0\end{equation}
as $R\to0$.

It remains to understand what happens with $v_\theta$. Indeed, we have
$$\eta^3v_\theta(r,x_3,t)=\int\limits^{x_3}_{-1}(\eta^3v_\theta)_{,3}(r,y,t)dy
$$ and thus 
$$\sup\limits_{-1|<x_3<1}|\eta^3v_\theta(r,x_3,t)|\leq c\Big(\int\limits^{1}_{-1}|(\eta^3v_\theta)_{,3}|(r,y,t)|^\frac {10}3dy\Big)^\frac 3{10}.
$$ The latter inequality can be re-written so that
$$\sup\limits_{-1|<x_3<1}\frac 1r|\eta^3v_\theta(r,x_3,t)|\leq c\Big(\int\limits^{1}_{-1}(|\Phi\eta^3|+|(\eta^3)_{,3}|v_\theta|/r)^\frac {10}3(r,y,t)dy\Big)^\frac 3{10}.
$$
Taking into account the inequality $|v_\theta|/r\leq |\nabla v|$ and  boundedness of $|\nabla v|$ in the support of $\nabla \eta$, after integration by parts, we  get:
$$\int\limits^0_{-1}\int\limits^1_0\Big(\sup\limits_{-1|<x_3<1}\frac 1r|\eta^3v_\theta(r,x_3,t)|\Big)^\frac {10}3rdr dt\leq c\|\Phi\eta^3\|^\frac {10}3_{\frac {10}3,Q}+C(v,\eta)<\infty.$$
So, for sufficiently small $R>0$, we have
$$\frac 1{R^2}\int\limits_{Q(R)}|v_\theta|^3dz\leq c\frac 1{R^\frac 32}\Big[\int\limits_{Q(R)}|v_\theta|^\frac {10}3dz\Big]^\frac 9{10}\leq $$
$$\leq c\frac 1{R^\frac 32}
\Big[\int\limits_{-R^2}^0\int\limits^R_0 \Big(R^\frac {10}3\int\limits^R_{-R}|v_\theta/r|^\frac {10}3dx_3\Big)rdr dt\Big]^\frac 9{10}\leq $$
$$\leq c
\frac {R^{(\frac{10}{3}+1)\frac{9}{10}}}{R^\frac 32}\Big[\int\limits^0_{-1}\int\limits^1_0\sup\limits_{-1<x_3<1}\Big|\frac {v_\theta\eta^3(r,x_3,t)}r\Big|^\frac {10}3rdr dt\Big]^\frac 9{10}\to 0
$$
as $R\to0$. According to the partial regularity theory for the Navier-Stokes equations and to (\ref{firstpart}), the origin $z=0$ is a regular point of $v$. Theorem \ref{mainresult} is proved.


\setcounter{equation}{0}
\section{Proof of Theorem \ref{PanType}}
Here, we are going to use two scale-invariant inequalities  proved in \cite{SS2018}. They are as follows: 
$$C(\varrho)\leq c\varepsilon\mathcal E(\varrho)+c_1(s,l,\varepsilon)G(\varrho),\quad G(\varrho)=(M^{s,l}(\varrho))^\frac 1{l(2-3/s-2/l)}$$
and 
$$\mathcal E(\theta \varrho)\leq \Big(\theta\mathcal E(\varrho)+\frac 1{\theta^2}C(\varrho)+\frac 1{\theta^\frac 43}C^\frac 23(\varrho)\Big),
$$
where $\varepsilon>0$, $0<\theta<1$ and $0<\varrho <1$, various scale-invariant quantities are defined as
$$\mathcal E(\varrho)=E(\varrho)+A(\varrho)+D(\varrho),\quad C(\varrho)=\frac1{\varrho^2}\int\limits_{Q(\varrho)}|v|^3dz,$$
$$E(\varrho)=\frac 1\varrho\int\limits_{Q(\varrho)}|\nabla v|^2dz,
\quad A(\varrho)=\frac 1{\varrho}\sup\limits_{-\varrho^2<t<0}\int\limits_{\mathcal C(\varrho)}|v(x,t)|^2dx,
$$
$$D(\varrho)=\frac 1{\varrho^2}\int\limits_{Q(\varrho)}|q|^\frac 32dz, \quad M^{s,l}(\varrho)=\frac 1{\varrho^\kappa}\int\limits_{-\varrho^2}^0dt\Big(\int\limits_{\mathcal C(\varrho)}|v|^sdx\Big)^\frac ls,
$$

$\kappa=l\Big(\frac 3s+\frac 2l-1\Big)$ and numbers $s$ and $l$ satisfies restrictions:
$$\frac 12>\frac 3s+\frac 2l-\frac 32>\max\Big\{\frac 1{2l},\frac 12 -\frac 1l\Big\}.$$

From the above inequalities, it can be easily derived
$$\mathcal E(\theta\varrho)\leq c\Big[\Big(\theta+\frac \varepsilon{\theta^2}+\frac {\varepsilon^\frac 12}{\theta}\Big)\mathcal E(\varepsilon) +c_2(\varepsilon,\theta)+c_3(s,l,\varepsilon,\theta)G(\varepsilon)\Big].
$$ Next,  we choose  a positive number $\theta$ to satisfy the inequality $c\theta \leq 1/4$ and then pick up a positive number $\varepsilon $ so that 
$$\frac \varepsilon{\theta^2}+\frac {\varepsilon^\frac 12}\theta<\frac 14.$$
Hence, we have
$$\mathcal E(\theta\varrho)\leq \frac 12\mathcal E(\varrho)+c+c_5(s,l)G(\varrho).$$
After iterations, we get the inequality
$$\mathcal E(\theta^k\varrho)\leq \frac 1{2^k}\mathcal E(\varrho)+c+c_5(s,l)\sum\limits^{k-1}_{i=0}\frac {G(\theta^i\varrho)}{2^{k-1-i}}$$ being valid for each natural $k$.

Now,  assume that condition (\ref{Pan})  holds,    
 set $s=\frac 74$ and $l=10$,  and try to evaluate the quantity $M^{s,l}$ and thus the quantity $G$. So, we have
$$
G(\varrho)\leq \frac 1{\varrho^\frac{64}7}\int\limits^0_{-\varrho^2}dt\Big(2\pi_0\int\limits^\varrho_{-\varrho}dx_3\int\limits^\varrho_0\Big(\frac cr\ln^\alpha \ln\frac 1r\Big)^\frac 74rdr\Big)^{
\frac {40}7
}.
$$
To estimate the above integral, we assume that a number $\alpha$ and the variable $\varrho$ are positive  and sufficiently small. 
Then,  integration by parts gives us: 
$$\int\limits^\varrho_0\Big(\frac 1r\ln^\alpha \ln\frac 1r)^\frac 74 r dr=4\varrho^\frac 14\ln^\frac {7\alpha}4\ln\frac 1\varrho+7\alpha\int\limits^\varrho_0\frac 1{\varrho^\frac 34}\frac 1{\ln^{1-\frac 74\alpha}\ln\frac 1\varrho}\frac 1{\ln\frac 1\varrho}dr \leq 
$$
$$\leq c\varrho^\frac 14\ln^\frac {7\alpha}4\ln\frac 1\varrho.$$
As a result,  
$$G(\varrho)\leq c\ln^{\frac{35}3\alpha}\ln\frac 1\varrho.
$$
Therefore,
$$\mathcal E(\theta^k\varrho)\leq \frac 1{2^k}\mathcal E(\varrho)+c+c\ln^{\frac{35}3\alpha}\ln\frac 1{\theta^k\varrho}.
$$ From the last estimate, it follows that
$$\mathcal E(\varrho) \leq C(\mathcal E( 1/2),\alpha)\ln^{\frac{35}3\alpha}\ln 1/{\varrho}
$$ for all $0<\varrho\leq 1/2$.
It remains to notice $f(R)+M(R)\leq c \sqrt{\mathcal E(R)}$. Theorem \ref{PanType} has been proved.

\setcounter{equation}{0}
\section{Corrections to the preprint \cite{Seregin2021}
}

In the note \cite{Seregin2021},
there is an error in calculations of a certain integral at the very end of Proposition 1.4. It can be corrected so that all results of  note \cite{Seregin2021}  remain  true. The correct version of Proposition 1.4  of \cite{Seregin2021} is Proposition \ref{modcon} of the present note. Keeping notation of the note \cite{Seregin2021}, let us comment to changes  to the proof.

The first comment is related to more delicate estimate of $\hat\beta_2$ and is as follows.

Obviously, there exists a number $0<R_{*5}(M_0,c_*,\alpha)\leq \min\{1/6,R_{*2}\}$ such that
$$2c(M_0,c_*)\ln^{224\alpha-1}\sqrt{\ln \frac1{R}}\leq \frac 12$$
and 
$$c(M_0,c_*)\ln^{224\alpha}\sqrt{\ln \frac1{R}}\geq
\ln g(2R) 
$$
for $0<R\leq R_{*5}(M_0,c_*,\alpha)$ and thus
$$-c(M_0,c_*)\ln^{224\alpha}\sqrt{\ln \frac1R}=$$$$=-2c(M_0,c_*)\ln^{224\alpha}\sqrt{\ln \frac1R}++c(M_0,c_*)\ln^{224\alpha}\sqrt{\ln \frac1R}\geq $$$$-\Big(2c(M_0,c_*)\ln^{224\alpha-1}\sqrt{\ln \frac1R}\Big)\ln\sqrt{\ln \frac1R}+\ln g(2R)\geq$$$$\geq -\ln\Big(\ln \frac1R\Big)^\frac 14+\ln g(2R)
.$$
Now, the number $\hat\beta_2$ is estimated as follows:
\begin{equation}
	\label{hatbeta2}
\hat\beta_2\geq \Big(\ln\frac 1R
\Big)^{-\frac 14}g(2R)
\end{equation} for $0<R\leq R_{*5}(M_0,c_*,\alpha)$. 
Then $$\beta(2R)=\frac c{\ln^\frac 34(1/R)}$$
and estimate of $\eta_k$ is as follows:
$$\leq -c\sum\limits^k_{i=0}(\ln ( 2^{2i-1}/R))^{-\frac 34}
=-c\sum\limits^k_{i=0}\frac 1{(i\ln 4+\ln (1/{(2R)}))^\frac 34
}\leq$$
$$\leq-c\int\limits^{k+1}_0\frac {dx}{(x\ln 4+\ln( 1/(2R))
)^\frac 34}=$$$$=-\frac {4c}{\ln 4}(x\ln 4+\ln(1/(2R)))^\frac 14\Big|^{k+1}_0=
$$$$=-c\Big(\ln^\frac 14(2^{2k+1}/R)
-\ln^\frac 14(1/(2R))\Big)
\Big).$$
So, (\ref{osc}) follows. 

Changes in the proof of Theorem 1.3 in \cite{Seregin2021} are based on the inequality
$$|\sigma (\varrho,x_3,t)|\leq C(c_*,\alpha)e^{-c\ln^\frac 14(1/(2\varrho))}\Sigma_0\leq C(c_*,\alpha)\Sigma_0\frac {m!}{c^m\ln^{\frac m4}(1/(2\varrho))},$$
being valid for all natural numbers $m$, and references to \cite{Wei2016}, \cite{LeiZhang2017}, and \cite{LiPan2019}.

\end{document}